\documentclass[11pt]{article}
\PassOptionsToPackage{obeyspaces}{url}
\usepackage{amsfonts, amsmath, amssymb}

\usepackage[hidelinks]{hyperref}
\usepackage{bookmark}
\bookmarksetup{open,numbered,depth=5} 

\usepackage{amsthm}

\usepackage{breakurl}
\usepackage{comment}
\usepackage{tikz}

\usepackage{etoolbox} 
\patchcmd{\thebibliography}{\leftmargin\labelwidth}{\leftmargin\labelwidth\addtolength\itemsep{-0.1\baselineskip}}{}{}

\usepackage{mathtools}

\author{Boris Bukh\thanks{Department of Mathematical Sciences, Carnegie Mellon University, Pittsburgh, PA 15213, USA\@. Supported in part by U.S.\ taxpayers through NSF grant DMS-2154063.}
\and 
Alexey Vasileuski\thanks{Department of Mathematical Sciences, Carnegie Mellon University, Pittsburgh, PA 15213, USA\@. Supported in part by U.S.\ taxpayers through NSF grants DMS-2154063 and CAREER DMS-2042428.}\\}

\title{New bounds for the same-type lemma}

\usepackage[nameinlink]{cleveref}

\newtheorem{theorem}{Theorem}
\newtheorem{lemma}[theorem]{Lemma}

\newtheorem{proposition}[theorem]{Proposition}

\newcommand*{\eqdef}{\stackrel{\mbox{\normalfont\tiny def}}{=}}  
\newcommand*{\veps}{\varepsilon}                                 
\DeclarePairedDelimiter\abs{\lvert}{\rvert}                      

\newcommand*{\R}{\mathbb{R}}
\DeclareMathOperator{\conv}{conv}
\DeclareMathOperator{\poly}{poly}
\DeclareMathOperator{\sign}{sign}
\DeclareMathOperator{\orient}{orient}

\begin{document}
\maketitle
\begin{abstract}
  Given finite sets $X_1,\dotsc,X_m$ in $\R^d$ (with $d$ fixed), we prove that there are respective subsets $Y_1,\dotsc,Y_m$ with $\abs{Y_i}\geq \frac{1}{\poly(m)}\abs{X_i}$
  such that, for $y_1\in Y_1,\dotsc,y_m\in Y_m$, the orientations of the\linebreak $(d+1)$-tuples from $y_1,\dotsc,y_m$ do not depend on the actual choices of points $y_1,\dotsc,y_m$.
  This generalizes previously known case when all the sets $X_i$ are equal. Furthermore, we give a construction showing that polynomial dependence on $m$ is unavoidable,
  as well as an algorithm that approximates the best-possible constants in this result.
\end{abstract}

\section{Introduction}

We say that the sets $Y_1,\dotsc,Y_{d+1}$  in $\R^d$ have the \emph{same-type property} if, for every choice of points $y_1\in Y_1,\dotsc,y_{d+1}\in Y_{d+1}$,
the orientation of points $y_1,\dotsc,y_{d+1}$ is the same. More generally, we say that the sets $Y_1,\dotsc,Y_m$ have the same-type property if every $d+1$
of them do. A natural question is the following: given disjoint finite sets $X_1, \dots, X_m$ in $\R^d$ such that their union is in general position, are there large subsets $Y_i \subseteq X_i$ such that the sets $Y_1, \dots, Y_r$ have the same-type property?
B{\'a}r{\'a}ny and Valtr  \cite{barany1998positive} proved that each $Y_i$ may be taken to have a positive fraction of points from the corresponding $X$-set.
How large could this fraction be?
Formally, for disjoint sets $X_1, \dots , X_m$ in $\R^d$, whose union is in general position, denote by $c(X_1,\dots, X_m)$ the largest constant $c$ for which there exist $Y_1, \dots, Y_n$ having the same-type property and satisfying $Y_i \subseteq X_i$, $|Y_i| \ge c |X_i|$.
For fixed number $m$ and dimension $d$, denote by $c(m,d)$  the infimum of $c$ for all such  configurations. 

The same-type lemma has been used to prove a number of positive fraction results in discrete geometry,
including Radon theorem, Tverberg theorem and Erd\H{o}s--Szekeres theorem \cite{barany1998positive}. Notably, a quantitative version of the latter
due to P\'or and Valtr \cite{por_valtr} was a crucial ingredient in Suk's proof \cite{suk_convex}
of the bound $2^{n+o(n)}$ for the number of points on plane guaranteeing the existence of $n$ points in convex position.
Additional results that directly use the same-type lemma are \cite{kt_application,fh_application}; also, many arguments
that are similar to the same-type lemma appear in the literature, e.g.,
\cite{pach_multicolored,bukh_hubard,pach_solymosi_segments,fox2016polynomial,mirzaei_suk}.

In their original paper, B{\'a}r{\'a}ny and Valtr showed that 
$c(m,d)$ is at least $(d+1)^{(-2^d-1)\binom{m-1}{d}}$. Fox,~Pach and Suk  \cite{fox2016polynomial} improved this  to $c(m,d) \ge 2^{-O(d^3 m \log m)}$.

Our first result shows that $c(m,d)$ is polynomial in $m$, for fixed $d$. 
\begin{theorem}
    \label{theorem:maintheorem}
    For $d \ge 2$ and $m \ge d$ the constant $c(m,d)$  satisfies
     \[ 
     d^{-50d^3}m^{-d^2} \le c(m,d) \le  d^dm^{-d}.
     \] 
\end{theorem}
Polynomial bounds were previously known only for the special case when the sets $X_1,\dotsc,X_m$ are all equal
(see Lemma 3.2 in \cite{mirzaei_suk}, with references to \cite{fox2016polynomial}). Our upper bound
of $d^d m^{-d}$ applies even to this special case; it is a first upper bound both in the special and general
cases.

We also show that the constants $c(m,d)$ can be computed with arbitrary precision, at least in principle.
\begin{theorem}
\label{theorem:approximations_main}
  There exists an algorithm that computes, for any input $m,d\in \mathbb{N}$ and $\veps>0$, a constant
  $c'(m,d)$ satisfying $\abs{c'(m,d)-c(m,d)}<\veps$.
\end{theorem}


\paragraph{Thanks.} We are grateful to Amzi Jeffs for helpful discussions.

\section{Preliminaries}
\paragraph{Sets with the same-type property.}
For simplicity, we say that a family of sets $X_1, \dotsc, X_m \subseteq \R^d$ is in \emph{general position} if the sets $X_1, \dots, X_m$ are disjoint and their union is a set of points in general position.
We start with a convenient sufficient condition for sets to have the same-type property.

\begin{lemma}
    \label{lemma:SameTypeFromConnected}
    Suppose that $Y_1, \dots,  Y_{d+1} \subseteq \R^d$ are connected sets and no hyperplane intersects all of them.  Then the sets $Y_1, \dots , Y_{d+1}$ have the same-type property.
\end{lemma}
\begin{proof}
    Note that, since the set $Y\eqdef Y_1\times \dotsb\times Y_{d+1}$ is a product of connected sets, it is itself connected.
    Suppose that the sets $Y_1,\dotsc,Y_{d+1}$ lack the same-type property. Then the
    sets
    \begin{align*}
      Y_+&\eqdef \{(y_1,\dotsc,y_{d+1})\in Y : \orient(y_1,\dotsc,y_{d+1})>0\},\\
      Y_-&\eqdef \{(y_1,\dotsc,y_{d+1})\in Y : \orient(y_1,\dotsc,y_{d+1})<0\}
    \end{align*}
    are both non-empty. Since both $Y_+$ and $Y_-$ are relatively open in $Y$, and $Y$ is connected, this implies
    that $Y_+\cup Y_-\neq Y$, i.e., there exists $(y_1,\dotsc,y_{d+1})\in Y$ such that the points
    $y_1,\dotsc,y_{d+1}$ are coplanar.
\end{proof}

The following is a converse to \Cref{lemma:SameTypeFromConnected} under slightly different conditions.

\begin{lemma}
\label{lemma:SameTypeFailing}
Suppose that the family of sets $Y_1, \dots,  Y_{d+1} \subset \R^d$ is in general position. If some hyperplane intersects each of the sets $\conv{Y_i}$, then $Y_1, \dots , Y_{d+1}$ do not have the same-type property.
\end{lemma}
\begin{proof}
  Let $H$ be any such hyperplane. We may assume that the sets $Y_1,\dotsc,Y_{d+1}$ are finite. Indeed, using Carath\'eodory's theorem,
  we may replace $Y_i$ by a subset of at most $d+1$ points whose convex hull contains a point of $H\cap \conv{Y_i}$.

  Keeping the condition $H\cap \conv{Y_i}\neq \emptyset$,
  perturb $H$ so that it contains points of $d$ sets among $Y_1, \dots, Y_{d+1}$, say the points
  $y_1\in Y_1,\dotsc,y_d\in Y_d$.  Since $Y_1\cup\dots\cup Y_{d+1}$ is in general position,
  it follows that $H$ contains no point of $Y_{d+1}$. Because $H$ does intersect $\conv{Y_{d+1}}$, the set
  $Y_{d+1}$ contains points $y_{d+1}^{\vphantom{'}}$ and $y_{d+1}'$ that lie on the opposite sides of $H$.
  So, the orientations of the tuples $(y_1^{\vphantom{'}},\dotsc,y_d^{\vphantom{'}},y_{d+1}^{\vphantom{'}})$ and $(y_1^{\vphantom{'}},\dotsc,y_d^{\vphantom{'}},y_{d+1}')$
  are opposite, which contradicts the same-type property of $Y_1,\dotsc,Y_{d+1}$.
\end{proof}

Also, in further proofs it will be useful for us to have enough points in each set. The following lemma implies that small constructions can be blown up to arbitrarily large size,
with no impact on the same-type constant $c(\,\cdots)$.

\begin{lemma}
    \label{lemma:pointsBlowing}
    Suppose that $X_1, \dots, X_m$ in $\R^d$ is a family of sets in general position. Denote by $X_1^{(n)}$ the set obtained by replacing each point of $X_1$ by cloud of $n$ points lying close enough to the original and preserving general position. Then,
    \[
    c(X_1^{\vphantom{()}}, \dots, X_m^{\vphantom{()}}) = c(X_1^{(n)}, X_2^{\vphantom{()}}, \dots, X_m^{\vphantom{()}}). 
    \]
\end{lemma}
\begin{proof}
  First, we prove the inequality $c(X_1^{\vphantom{()}}, \dots , X_m^{\vphantom{()}}) \le c(X_1^{(n)}, \dots, X_m^{\vphantom{()}})$. Consider arbitrary subsets $Y_1 \subseteq X_1, \dots, Y_m \subseteq X_m$ having the same-type property. Take the subset $Y_1' \subseteq X_1^{(n)}$, consisting of all points of clouds corresponding to the points of $Y_1$.
  Since the  cloud are sufficiently small,  $Y_1', Y_2^{\vphantom{''}},\dots, Y_m^{\vphantom{''}}$ also have the same-type property and $\abs{Y_1} / \abs{X_1} = \abs{Y_1'} / \abs{X_1^{(n)}}$. 

  Next, we prove the inequality $c(X_1^{\vphantom{()}}, \dots ,X_m^{\vphantom{()}}) \ge c(X_1^{(n)}, \dots, X_m^{\vphantom{()}})$. Suppose that the subsets $(Y_1', Y_2^{\vphantom{'}}, \dots, Y_m^{\vphantom{'}})$ of $(X_1^{(n)}, X_2^{\vphantom{()}}, \dots, X_m^{\vphantom{()}})$ have the same-type property.
  Define
  \[
    Y_1\eqdef \{x\in X_1 : Y_1'\text{ contains a point of the cloud around }x\}. 
  \]
  If each cloud lies sufficiently near the original point, the sets $Y_1, \dots, Y_m$ have the same-type property, and $|Y_1| / |X_1|$ is at least $|Y_1'| / |X_1^{(n)}|$.
\end{proof}

 This lemma implies that, in the definition of $c(m,d)$ it is enough to consider only the sets of the same size, which may be assumed to exceed an arbitrarily large constant.
 
\paragraph{Polynomial partitioning.}
 For the proof of the lower bound  we use the polynomial partitioning introduced by Guth and Katz. Since the proof in \cite{distanceNumber} does not track the dependence on $d$, we include the relevant calculation. The next lemma is a version of Theorem 4.1 in \cite{distanceNumber} with a fully explicit bound. 
\begin{lemma}
\label{lemma:PolynomialPartitioning}
    If $X$ is a set of $n$ points in $\R^d$ and $J \ge 1$ is an integer, then there is a polynomial surface $Z$ of degree $D \le 3d^2 2^{J/d}$  with the following property: each connected component of $\R^d \setminus Z$  contains at most $2^{-J} n$ points of $X$.
\end{lemma}

We rely on the polynomial ham sandwich theorem. We say that the real algebraic hypersurface $\{x\in \R^d : f(x)=0\}$ \emph{bisects} a point set $X$ if
both sets $\{ x \in \R^d : f(x) > 0 \}$ and $\{x \in \R^d : f(x) < 0 \}$ contain at most half of the points of~$X$. 

\begin{lemma}[Corollary 4.3 in \cite{distanceNumber}]
\label{lemma:findpolynom}
    Let $X_1, \dots , X_M$ be finite sets of points in $\R^d$ with $M = \binom{D+d}{d} - 1$. Then there is a real algebraic hypersurface of degree at most $D$ that bisects each $X_i$.
\end{lemma}

\begin{proof}[Proof of the \Cref{lemma:PolynomialPartitioning}.]
  Given polynomials $p_1,\dotsc, p_j$ and a sign vector $\veps=(\veps_1,\dotsc,\veps_j)\in \{-1,+1\}^j$,
  consider the cell of $\R^d$ on which the first $j$ polynomials have these signs, i.e.,
  \[
    C_{\veps}\eqdef \{x\in \R^d : \sign p_1(x)=\veps_1,\dotsc,\sign p_j(x)=\veps_j\}
  \]
  Write $X_{\veps}\eqdef X\cap C_{\veps}$.
  
  We claim that there are polynomials $p_1, \dots, p_J$ of degrees $\deg p_{j+1}\leq d2^{j/d}$
  such that $\abs{X_\veps}\leq \abs{X}2^{-j}$ for every $j\leq J$ and every $\veps\in  \{-1,+1\}^j$.

  We find such polynomials one by one. Suppose that the first $j$ polynomials $p_1, \dots, p_j$ have been defined,
  and that all $2^j$ sets $X_{\veps}$ with $\veps\in \{-1,+1\}^j$ satisfy the condition above.
  Observe the inequality
 \[
    \binom{d 2^{j/d} + d}{d} = \prod_{i=0}^{d-1} \frac{d2^{j/d}+ d - i}{d-i} > \prod_{i=0}^{d-1}2^{j/d} = 2^j.
 \]
Taking this into account, \Cref{lemma:findpolynom} allows us to find a polynomial $p_{j+1}$ of degree at most $d 2^{j/d} + 1$ whose zero set bisects each~$X_{\veps}$. \medskip

Let $p$ be the product of $p_1, \dots, p_J$ and let $Z$ be its zero set, we claim that $Z$ is the desired hypersurface. The degree of $p$ is at most
\[
 \sum_{j=0}^{J-1} \left(d2^{j/d}+1\right) = d \left(1+\frac{2^{J/d}-1}{2^{1/d}-1}\right) \le  3 d^2 2^{J/d}.
\]

Since each connected component of $\R^d\setminus Z$ is a subset of some $C_{\veps}$, and therefore contains at most $2^{-J}\abs{X}$ points of $X$, the lemma follows.
\end{proof}

Also, we we will need the following bound due to Warren. One of its consequences is that
the number of parts into which the surface from \Cref{lemma:PolynomialPartitioning} cuts $\R^d$
is only slightly larger than~$2^J$.

\begin{lemma}[Lemma 6.2 in \cite{hyperplanebound}]
\label{lemma:componentsnumberbound}
    Let $f$ be a real polynomial of degree $D$ in $d$ variables. Then the number of connected components of $\R^d \setminus Z(f)$ is at most $6(2D)^d$.
\end{lemma}

\section{Proof of \texorpdfstring{\Cref{theorem:maintheorem}}{Theorem 1}: the lower bound}

Consider any family $X_1, \dots, X_m \subseteq \R^d$ in general position. Thanks to \Cref{lemma:pointsBlowing}, we may assume that all of them have the same size $n$, which is sufficiently large.  Also, since slight perturbations of points do not change the orientation, by perturbing the sets generically we may additionally assume that, for any $D$, no polynomial surface of degree $D$ intersects more than $\binom{D+d}{d}-1$ points.  

Fix $r \eqdef m^{d^2}d^{30d^3}$. For each $i$ apply \Cref{lemma:PolynomialPartitioning} with $J =  \lceil \log_2{r} \rceil $ to obtain polynomial surface $Z_i$ of degree at most $D_0\eqdef 6d^2 r^{1/d} $ such that each connected component of $\R^d \setminus Z_i$  contains at most  $n / r$ points of $X_i$. By \Cref{lemma:componentsnumberbound}, the total number of such components is at most $k \eqdef 6\cdot12^dd^{2d}r \le m^{d^2} d^{50d^3} / 4$. Denote by $\mathcal{C}_i$ the set of these components. Having taken $n$ large enough, we observe that
 each $Z_i$ contains at most $\binom{D_0+d}{d}-1\leq n/2$ points of $X_i$.
Some components have at most $n/4k$ points of $X_i$; these account for at most $n/4$ points of $X_i$ in all. Let $\mathcal{C}_i'$
be the set of components containing more than $n/4k$ points. Then, $\abs{\mathcal{C}_i'}\geq  (n - n/4- n/2)/(n/r) = (n/4)/(n/r)=r/4$.

Next, we define an auxiliary $(d+1)$-uniform $m$-partite hypergraph $H$ with parts $\mathcal{C}_1',\mathcal{C}_2',\dotsc,\mathcal{C}_m'$.
Then, for any distinct $i_1,i_2,\dotsc,i_{d+1} $ and any $C_{i_1}\in \mathcal{C}_{i_1}',C_{i_2}'\in \mathcal{C}_{i_2}',\dotsc,C_{i_{d+1}}\in \mathcal{C}_{i_{d+1}}'$
put an edge between vertices $C_{i_1},C_{i_2},\dotsc,C_{i_{d+1}}$ if and only if these components can be pierced by a single hyperplane.
The key observation is that the hypergraph $H$ is sparse, thanks to the polynomial partitioning.

\begin{lemma}
    \label{lemma:edgenumberbound}
    Any $d+1$ parts of $H$ span at most $d^{20d^2}r^{d+1-1/d}$ transversal edges. 
\end{lemma}
\begin{proof}
    Without loss of generality, consider parts indexed by $1, \dotsc, d+1$. Denote the polynomials defining $Z_i$ by $f_i$.
    
    Consider the following $d+1$ linear maps from $\R^{d^2+d-1}=(\R^d)^d\times (\R^1)^{d-1}$ to $\R^d$.
     \begin{align*}
         l_1(x_1, \dots, x_d, \alpha_1, \dots, \alpha_{d-1}) &\eqdef x_1, \\
         l_2(x_1, \dots, x_d, \alpha_1, \dots, \alpha_{d-1}) &\eqdef x_2, \\
         \vdots\qquad\qquad\quad&                 \\
         l_d(x_1, \dots, x_d, \alpha_1, \dots, \alpha_{d-1}) &\eqdef x_d, \\
         l_{d+1}(x_1, \dots, x_d, \alpha_1, \dots, \alpha_{d-1}) &\eqdef \sum_{i=1}^{d-1} (\alpha_i x_i) + \Bigl(1 - \sum_{i=1}^{d-1}\alpha_i\Bigr)x_d.\\
    \intertext{Denote by $Z$ the algebraic hypersurface in $\R^{d^2+d-1}=(\R^d)^d\times (\R^1)^{d-1}$ defined by the  polynomial} 
    h(x_1,  \dots, x_d, \alpha_1, \dots , \alpha_{d-1})  &\eqdef \prod_{i=1}^{d+1}f_i\Big(l_i(x_1, \dots, x_d, \alpha_1, \dots, \alpha_{d-1}) \Big),
    \end{align*}
    and denote by $\mathcal{C}$ the set of corresponding connected components.

    By the definition of $h$, the image of $\R^{d^2+d-1} \setminus Z$ under any $l_i$ belongs to $\R^{d} \setminus Z_i$.  Moreover, since $l_i$ is continuous and the image of a connected set under a continuous map is connected, the image of any set in $\mathcal{C}$ contained in exactly one of the sets in $\mathcal{C}_i$. This way $\ell_1\times \ell_2\times \dots\times \ell_{d+1}$ induces a well-defined map $L\colon \mathcal{C} \to \mathcal{C}_1 \times \dots \times \mathcal{C}_{d+1}$. 
    
    We observe that any tuple  $(C_1,\dotsc,C_{d+1})\in \mathcal{C}_1'\times \dots \times \mathcal{C}_{d+1}' \subseteq \mathcal{C}_1\times \dots \times \mathcal{C}_{d+1}$ that forms an edge in $H$ is in the image of~$L$. Indeed, the sets $C_1,C_2,\dotsc,C_{d+1}$ are open; so if some hyperplane pierces them, one can find points $x_1 \in C_1, x_2\in C_2, \dotsc, x_{d+1} \in C_{d+1}$ such that $x_{d+1}$ is an affine combination of $x_1$ through $x_d$, say $x_{d+1}=\sum_{i=1}^d \alpha_i x_i$ with $\sum_{i=1}^d \alpha_i=1$. In this case, the image of $(x_1,\dotsc,x_d,\alpha_1,\dotsc,\alpha_{d-1})$ under $l_1\times l_2\times \dots\times l_{d+1}$ lies in $C_1\times C_2\times \dots\times C_{d+1}$.

    The reasoning above implies that the number of edges spanned by sets $\mathcal{C}_1',\dotsc,\mathcal{C}_{d+1}'$ in $H$ does not exceed $\abs{\mathcal{C}}$. Since polynomial $h$ depends on $d^2+d-1$ variables and has degree at most $12d^3r^{1/d}$, \Cref{lemma:componentsnumberbound} gives the bound of $6 (24d^3)^{d^2+d-1} r^{(d^2+d-1)/d} \le d^{20d^2}r^{d+1-1/d}$. 
\end{proof}

From each $\mathcal{C}_i'$ pick an element $C_i$ independently at random. We claim that, with positive probability, the $d+1$ vertices $C_1,\dotsc,C_{d+1}$ form an independent set in $H$.
This would imply that the sets $C_1\cap X_1,C_2\cap X_2,\dotsc,C_{d+1}\cap X_{d+1}$ have the same-type property. Since each of these has at last $n/4k$ elements, that would conclude the
proof.

The claim follows from Lov\'asz Local Lemma. Indeed, for the set $I \subset [m]$ of size $d+1$ denote by $\mathcal{B}_I$ the event that vertices $C_i \in \mathcal{C}_i'$ with $i \in I$ form an edge in $H$. Since $\abs{\mathcal{C}_i'}\geq r/4$,
\Cref{lemma:edgenumberbound} shows that probability of such event is at most $4^{d+1} d^{20d^2} r^{-1/d}$. If $I\cap J=\emptyset$, then the events $\mathcal{B}_I$ and $\mathcal{B}_J$
are defined by disjoint sets of random choices. Hence, the natural dependence graph has degree at most
$(d+1)\binom{m}{d}\leq (d+1)(me)^d / d^d$. Since with our choice of the constant $r$ we have
\[
e \frac{(d+1)(me)^d4^{d+1} d^{20d^2} }{r^{1/d}d^d} \le    \frac{d^{30d^2}m^d }{r^{1/d}} =1,
\]
the condition of the symmetric Local Lemma (see e.g., \cite[Corollary 5.1.2]{alon_spencer}) is satisfied and the event $\bigcap_I \mathcal{B}_I$ holds with positive probability.

\section{Proof of \texorpdfstring{\Cref{theorem:maintheorem}}{Theorem 1}: the upper bound}

To prove an upper bound on $c(r,m)$, we provide a series of constructions of arbitrarily large sets $X_1, \dots, X_m$ without large subsets with the same-type property. 
A set $P\subset \R^d$ of size $\binom{n}{d}$ is a \emph{grid set} if there exist $n$ hyperplanes in $\R^d$ whose set
of $d$-wise intersections is $P$. Our constructions will be suitable small perturbations of grid sets.
The purpose of the perturbation is to ensure general position.

Convex sets intersect the grid sets and their perturbations slightly differently. The next lemma says that
the difference is small, because the boundary of a convex set meets a grid set in a negligible fraction of points. 

Denote by $\partial C$ the boundary of a set $C\subset \R^d$.
\begin{lemma}
\label{lemma:boundarypointsbound}
    Suppose that $P\subset \R^d$ is a grid set of size $\binom{n}{d}$. Then for any compact convex set $C$ we have $|\partial C \cap P| \le 2\binom{n}{d-1}$. 
\end{lemma}
\begin{proof} 
  We  prove this by induction on $d$. If $d = 1$, then $C$ is a segment, which has two boundary points. So, assume that $d>1$ and that the lemma holds for $d-1$ in place of~$d$.
  Let $H_1,\dotsc,H_n$ be the $n$ hyperplanes generating the grid set~$P$. Put $C_i\eqdef C\cap H_i$. Write $\partial C_i$ for the relative
  boundary of $C_i$ inside the hyperplane $H_i$.

  Consider an arbitrary point $x\in \partial C\cap P$. Since $x\in \partial C$, we can find a hyperplane $H$ passing through $x$ such that $C$ lies on one side of~$H$. Then, at least $d-1$ hyperplanes among $H_1,\dotsc,H_n$ pass through $x$ but differ from $H$.
  For each such $H_i$, the codimension-$2$ subspace $H_i \cap H$ contains $x$ and  bounds $C_i$ inside $H_i$, implying that $x\in \partial C_i$.
     
  Since this holds for every $x\in \partial C\cap P$, it follows that
    \begin{equation}
        |\partial C \cap P| (d-1)  \le  \sum_{i=1}^n |\partial C_i \cap P|.
    \end{equation}      

    Observe that $C_i\cap P$ is itself a grid set inside the $(d-1)$-dimensional hyperplane $H_i$.
    Therefore, bounding $|\partial C_i \cap P|$ by induction, we obtain
    \[
      |\partial C \cap P| \le \frac{n}{d-1}\cdot 2\binom{n-1}{d-2} = 2 \binom{n}{d-1}.\qedhere
    \]
\end{proof}

Fix a grid set $X$ of size $\binom{n}{d}$. Let $X_1,\dotsc,X_m$ be small perturbations of $X$ chosen so that the  family $X_1, \dots,  X_m$ is in general position.
We shall show that these $m$ sets do not contain large subsets with the same-type property.

For $x\in X_i$, write $P(x)$ for its \emph{predecessor}, the point of $X$ that $x$ is a perturbation of.
Similarly, write $P(Y)$ for the set of predecessors of a set $Y\subset X_i$.
Let $\mathcal{H}$ be the set of $n$ hyperplanes generating the grid set $X$.

Consider any sets $Y_1, \dots, Y_m$ with the same-type property such that $Y_i \subseteq X_i$.
Writing $\operatorname{int}{A}$ for the interior of a set $A$, define, for each $i=1,2,\dotsc,m$,
\[
  Z_i \eqdef X \cap \operatorname{int} \conv{Y_i}.
\]

Breaking the set $P(Y_i)$ into the boundary and the interior parts we obtain
\begin{align*}
  P(Y_i)&=\bigl(P(Y_i)\cap \operatorname{int} \conv{P(Y_i)}\bigr)\cup \bigl(P(Y_i)\cap \partial \conv{P(Y_i)}\bigr)\\
        &\subseteq \bigl(X\cap \operatorname{int} \conv{P(Y_i)}\bigr)\cup \bigl(X \cap \partial \conv{P(Y_i)}\bigr).\\
  \intertext{If the perturbation defining $Y_i$ is sufficiently small, $X\cap \operatorname{int} \conv{P(Y_i)}\subseteq Z_i$, and so}
  P(Y_i)&\subseteq Z_i \cup \bigl(X \cap \partial \conv{P(Y_i)}\bigr).
\end{align*}

Since $\abs{Y_i}=\abs{P(Y_i)}$, \Cref{lemma:boundarypointsbound} tells us that
\begin{equation}\label{eq:yzbound}
\begin{aligned}
  \abs{Y_i}&\leq \abs{Z_i}+\abs{X \cap \partial \conv{P(Y_i)}} \\&\leq \abs{Z_i}+o(\abs{X}) \quad\text{as }n\to\infty.
\end{aligned}
\end{equation}

Since $Z_i\subset \conv{Y_i}$, \Cref{lemma:SameTypeFailing} implies that no hyperplane in $\mathcal{H}$
intersects more than $d$ sets among $Z_1,\dotsc,Z_m$. 
By the pigeonhole principle, for some $i$, some set $Z_i$ intersects at most $dn/m$ hyperplanes of $\mathcal{H}$.
Hence, this $Z_i$ is contained in the grid set generated by $dn/m$ hyperplanes,
implying that
\[
  \abs{Z_i}\le \binom{dn/m}{d}=(d/m)^d\binom{n}{d}\bigl(1+o(1)\bigl)=(d/m)^d\abs{X}\bigl(1+o(1)\bigl)\quad\text{as }n\to\infty.
\]
From this and \eqref{eq:yzbound}, it follows that $c(m,d)\leq (d/m)^d$.

\paragraph{Remark.} The preceding argument also shows that the upper bound of $(d/m)^d$ holds for the non\nobreakdash-partite version
of the same-type lemma. Precisely, if $X$ is a slight perturbation of a grid set of size $\binom{n}{d}$, and
$Y_1,\dots,Y_m\subset X$ are disjoint sets with the same-type property, then at least one of $Y_i$ is of size at most $\binom{dn/m}{d}\bigl(1+o(1)\bigl)$.

\section{Arbitrarily good approximations to \texorpdfstring{$c(m,d)$}{c(m,d)}}
We now turn to the task of computing arbitrarily good approximations to $c(m,d)$. We use the following well-known
result of Vapnik and \v{C}ervonenkis about the existence of $\veps$-approximants.

 \begin{lemma}[Section 1.5 of \cite{vc_original}]
   \label{theorem:vc_approximations}
   Let $\mathcal{F}\subseteq 2^X$ be a set family of VC-dimension~$D$. Then, for any $0<\veps<1$ there exists a set $A\subseteq X$ of size at most
   $\frac{32}{\veps^2}D \ln \frac{16D}{\veps^2}$ such that
   \[
    \left| \frac{|F \cap X|}{|X|} - \frac{|F \cap A|}{|A|}  \right| \le \veps\qquad\text{for all }F\in\mathcal{F}.
   \]
\end{lemma}

 For the purpose of estimating $c(m,n)$, this allows us to limit the search to bounded-size families.
 \begin{proposition}
    For any natural numbers $m, d$ and any $\varepsilon > 0$ there exist sets $A_1,  \dots, A_m \subseteq \R^d$ of size bounded by a computable function of $m, d, \varepsilon$ such that $\abs{c(A_1, \dots,  A_m) - c( m, d )} \le \varepsilon$.
 \end{proposition}
 \begin{proof}
   Pick finite sets $X_1, \dots, X_m \subset \R^d$ in general position such that $c(X_1, \dots,  X_m) < c( m, d) + \varepsilon / 2$.  Let $\mathcal{F}$ be the family of open polytopes in $\R^d$ with at most $m$ facets, its VC-dimension is at most $O(d m \log m)$ (see \cite[Lemma 10.3.1 and Proposition 10.3.3]{lecturesDG}).
   Apply \Cref{theorem:vc_approximations} to each $X_i$ and $\mathcal{F}$ with $\varepsilon / 2$ in place of $\veps$ to obtain sets $A_i$ of size bounded by some function of $m,d,\veps$. 

   We claim that $c(m,d)\leq c(A_1, \dotsc, A_m)\leq c(m,d) + \varepsilon$. Since the former inequality follows from the definition of the constant $c(\,\cdots)$, it remains to show the latter one.
   To that end, consider arbitrary subsets $Y_1 \subseteq A_1, \dots, Y_m \subseteq A_m$ with the same-type property. By \Cref{lemma:SameTypeFailing}, for each $i$ the set $Y_i$ is separated from $Y_1 \cup \dots  \cup Y_{i-1} \cup Y_{i+1} \cup \dots \cup Y_m$ by some hyperplane, which we denote by~$H_i$. The hyperplanes $H_1,\dotsc,H_m$ form a hyperplane arrangement;
   each $Y_i$ is contained within a single cell of this arrangement, which we denote by $P_i$. Since $P_i$ is an intersection of $m$ halfspaces, each $P_i$ is an open polyhedron with at most $m$ facets. Observe that the sets $P_1,\dotsc,P_m$  have the same-type property; this implies that $P_1 \cap X_1,\dotsc,P_m\cap X_m$ have the same-type property as well,
   and so, for some $i$, we have
   \[
     \frac{\abs{P_i\cap X_i}}{\abs{X_i}}\leq c(m,d)+\veps/2.
   \]
   Therefore, for this value of $i$, we have
   \[
     \frac{|Y_i|}{|A_i|} \le \frac{|P_i \cap A_i|}{|A_i|} \le \frac{|P_i \cap X_i|}{|X_i|} + \varepsilon / 2 \le  c(m,d) + \varepsilon.
   \]
   Since the sets $Y_1,\dotsc,Y_m$ are arbitrary, this completes the proof.
 \end{proof}

 The approximability of the constants $c(m,d)$ now follows from the famous result of Tarski on the decidability of the theory of real closed fields
 \cite{tarski} (see for example \cite[Theorem 2.77]{basu_pollack_roy} for a modern exposition).
 Indeed, the value of $c(A_1,\dotsc,A_m)$ depends only on the orientations of $(d+1)$-tuples from $A_1\cup\dotsb\cup A_m$.
 The existence of $A_1,\dotsc,A_m$ with specified orientations of $(d+1)$-tuples 
 can be expressed as an existential sentence in the language of ordered fields. 
 This sentence is decidable by the aforementioned result of Tarski (though deciding
 existential sentences can be done more efficiently than general sentences; see, e.g., \cite[Algorithm 13.1]{basu_pollack_roy}).
 So, the largest value of $c(A_1,\dotsc,A_m)$, subject to $\abs{A_i}=B_i$, can be computed by iterating over all possible sign patterns and
 checking if they are realizable by point sets in~$\R^d$.

\bibliographystyle{acm}
\bibliography{main}

\begin{thebibliography}{10}

\bibitem{alon_spencer}
{\sc Alon, N., and Spencer, J.~H.}
\newblock {\em The probabilistic method}, fourth~ed.
\newblock Wiley Series in Discrete Mathematics and Optimization. John Wiley \&
  Sons, Inc., Hoboken, NJ, 2016.

\bibitem{barany1998positive}
{\sc B\'{a}r\'{a}ny, I., and Valtr, P.}
\newblock A positive fraction {E}rd{\H{o}}s--{S}zekeres theorem.
\newblock {\em Discrete Comput. Geom. 19}, 3 (1998), 335--342.

\bibitem{basu_pollack_roy}
{\sc Basu, S., Pollack, R., and Roy, M.-F.}
\newblock {\em Algorithms in real algebraic geometry}, vol.~10 of {\em
  Algorithms and Computation in Mathematics}.
\newblock Springer-Verlag, Berlin, 2003.

\bibitem{bukh_hubard}
{\sc Bukh, B., and Hubard, A.}
\newblock Space crossing numbers.
\newblock {\em Combin. Probab. Comput. 21}, 3 (2012), 358--373.

\bibitem{fh_application}
{\sc Fabila-Monroy, R., and Huemer, C.}
\newblock Carath\'{e}odory's theorem in depth.
\newblock {\em Discrete Comput. Geom. 58}, 1 (2017), 51--66.

\bibitem{fox2016polynomial}
{\sc Fox, J., Pach, J., and Suk, A.}
\newblock A polynomial regularity lemma for semialgebraic hypergraphs and its
  applications in geometry and property testing.
\newblock {\em SIAM Journal on Computing 45}, 6 (2016), 2199--2223.
\newblock \arXiv{1502.01730}.

\bibitem{distanceNumber}
{\sc Guth, L., and Katz, N.~H.}
\newblock On the {E}rd{\H{o}}s distinct distances problem in the plane.
\newblock {\em Ann. of Math. (2) 181}, 1 (2015), 155--190.
\newblock \arXiv{1011.4105}.

\bibitem{hyperplanebound}
{\sc Kaplan, H., Matou\v{s}ek, J., and Sharir, M.}
\newblock Simple proofs of classical theorems in discrete geometry via the
  {G}uth-{K}atz polynomial partitioning technique.
\newblock {\em Discrete Comput. Geom. 48}, 3 (2012), 499--517.
\newblock \arXiv{1102.5391}.

\bibitem{kt_application}
{\sc K\'{a}rolyi, G., and T\'{o}th, G.}
\newblock Erd{\H{o}}s-{S}zekeres theorem for point sets with forbidden
  subconfigurations.
\newblock {\em Discrete Comput. Geom. 48}, 2 (2012), 441--452.

\bibitem{lecturesDG}
{\sc Matou\v{s}ek, J.}
\newblock {\em Lectures on discrete geometry}, vol.~212 of {\em Graduate Texts
  in Mathematics}.
\newblock Springer-Verlag, New York, 2002.

\bibitem{mirzaei_suk}
{\sc Mirzaei, M., and Suk, A.}
\newblock A positive fraction mutually avoiding sets theorem.
\newblock {\em Discrete Math. 343}, 3 (2020), 111730, 6.

\bibitem{pach_multicolored}
{\sc Pach, J.}
\newblock A {T}verberg-type result on multicolored simplices.
\newblock {\em Comput. Geom. 10}, 2 (1998), 71--76.
\newblock \arXiv{math/9603211}.

\bibitem{pach_solymosi_segments}
{\sc Pach, J., and Solymosi, J.}
\newblock Crossing patterns of segments.
\newblock {\em J. Combin. Theory Ser. A 96}, 2 (2001), 316--325.

\bibitem{por_valtr}
{\sc P\'{o}r, A., and Valtr, P.}
\newblock The partitioned version of the {E}rd{\H{o}}s-{S}zekeres theorem.
\newblock vol.~28. 2002, pp.~625--637.
\newblock Discrete and computational geometry and graph drawing (Columbia, SC,
  2001).

\bibitem{suk_convex}
{\sc Suk, A.}
\newblock On the {E}rd{\H{o}}s-{S}zekeres convex polygon problem.
\newblock {\em J. Amer. Math. Soc. 30}, 4 (2017), 1047--1053.

\bibitem{tarski}
{\sc Tarski, A.}
\newblock {\em A decision method for elementary algebra and geometry}.
\newblock University of California Press, Berkeley-Los Angeles, Calif., 1951.
\newblock 2nd ed.

\bibitem{vc_original}
{\sc Vapnik, V.~N., and \v{C}ervonenkis, A.~J.}
\newblock The uniform convergence of frequencies of the appearance of events to
  their probabilities.
\newblock {\em Teor. Verojatnost. i Primenen. 16\/} (1971), 264--279.
\newblock \url{http://mi.mathnet.ru/tvp2146} (English translation is at
  \url{https://epubs.siam.org/doi/10.1137/1116025}).

\end{thebibliography}
\end{document}